\documentclass[a4j,10pt]{article}

\usepackage{amsmath, amssymb, amsthm,amscd}
\usepackage{graphics, comment} 
\usepackage[dvips]{graphicx}

\allowdisplaybreaks

\numberwithin{equation}{section}
\theoremstyle{definition}
\newtheorem{example}{example}[section]
\newtheorem{df}[example]{Definition}
\theoremstyle{remark}
\newtheorem{remark}[example]{{\bf Remark}}
\theoremstyle{plain}
\newtheorem{lm}[example]{Lemma}
\newtheorem{prop}[example]{Proposition}

\newtheorem{ex}[example]{Example}

{\par\ifdim\lastskip<\bigskipamount\removelastskip\fi\smallskip%
\noindent\begin{trivlist}%
\item[\hskip\labelsep{\bf Proof\ }]}%
{\rule[-.1mm]{1.5mm}{2ex}\end{trivlist}\medskip\par}

\newtheorem{theorem}[example]{Theorem}

\newcommand{\lsp}{[\kern-0.15em[} 
\newcommand{\rsp}{]\kern-0.15em]} 

\def\Rc{\varrho} 

\theoremstyle{remark}

\begin{document}
\title{Nilsolitons of H-type in the Lorentzian setting} 
\author{Kensuke Onda and Phillip E. Parker\thanks{%
Onda was supported by funds of Nagoya University and OCAMI. \newline
2000 \emph{Mathematics Subject Classification:} 53C50, 53C21, 53C25.
\newline
\emph{Keywords and phrases:} Lorentzian Lie groups, left-invariant metrics,
algebraic Ricci solitons.} }
\maketitle


\begin{abstract}
\noindent
It is known that all left-invariant pseudo-Riemannian metrics on $H_3$ 
are algebraic Ricci solitons.
We consider generalizations of Riemannian $H$-type, namely pseudo$H$-type 
and $pH$-type.
We study algebraic Ricci solitons of left-invariant Lorentzian metrics on 
2-step nilpotent Lie groups of both types.
\end{abstract}

\section{\protect\LARGE Introduction and preliminaries }
The concept of an algebraic Ricci soliton was first introduced by Lauret
in the Riemannian case \cite{L01}. The definition extends to the
pseudo-Riemannian case:

\begin{df}
\textit{Let } $\left( \mathit{G,g}\right) $\textit{\ be a simply connected
Lie group equipped with a left-invariant pseudo-Riemannian metric $g$, and
let } ${\mathfrak{g}}$ denote \textit{the Lie algebra }of $\mathit{G.}$
\textit{Then $g$ is called an algebraic Ricci soliton if it satisfies }
\begin{equation}
\mathrm{Ric}=c\cdot \mathrm{Id}+D  \label{Soliton}
\end{equation}%
\textit{where }$\mathrm{Ric}$\textit{\ denotes the Ricci operator, $c$ is a
real number, and }$\mathrm{D\in Der}\left( \mathfrak{g}\right) $ ($\mathrm{D}
$ is \textit{a derivation of }$\mathit{\mathfrak{g}}$\textit{)}, that is:%
\begin{equation*}
\mathit{\mathrm{D}[X,Y]=[\mathrm{D}X,Y]+[X,\mathrm{D}Y]\ \,}\text{\ \textit{%
for any }}\mathit{\,\ X,Y\in \mathfrak{g}.}
\end{equation*}%
\textit{\ In particular, an algebraic Ricci soliton on a solvable Lie group
(a nilpotent Lie group) is called a solvsoliton (a nilsoliton).}
\end{df}

Obviously, Einstein metrics on a Lie group are algebraic Ricci solitons.

\begin{remark}
Let $G$ be a semisimple Lie group and $g$ a left-invariant Riemannian
metric. If $g$ is an algebraic Ricci soliton, then $g$ is Einstein
(\cite{L01}).
\end{remark}

Next we introduce Ricci solitons. Let $g_{0}$ be a pseudo-Riemannian metric
on a manifold $M^{n}$. If $g_{0}$ satisfies
\begin{equation}
\varrho \lbrack g_{0}]=c\cdot g_{0}+L_{X}g_{0}\ ,  \label{Ricci soliton}
\end{equation}%
where $\varrho $ denotes the Ricci tensor, $X$ is a vector field and $c$ is
a real constant, then $(M^{n},g_{0},X,c)$ is called a \textit{Ricci soliton
structure} and $g_{0}$ is \textit{the Ricci soliton.} Moreover, we say
that
the Ricci soliton $g_{0}$ is a \textit{gradient Ricci soliton} if the vector
field $X$ satisfies $X=\nabla f$ where $f$ is a function. The Ricci soliton
$g_{0}$ is said to be a \textit{non-gradient Ricci soliton} if the vector
field $X$ satisfies $X\neq \nabla f$ for any function $f$. If $c$ is
positive, zero, or negative, then $g_{0}$ is called a shrinking, steady, or
expanding Ricci soliton, respectively. According to \cite{CK04}, we check
that a Ricci soliton is a Ricci flow solution.

\begin{prop}[\protect\cite{CK04}]
\textit{A Riemannian metric $g_{0}$ is a Ricci soliton if and only
if $g_{0}$ is the initial metric of the Ricci flow equation}
\begin{equation*}
\frac{\partial }{\partial t}g(t)_{ij}=-2 \varrho \lbrack g(t)]_{ij}\ ,
\end{equation*}%
\textit{and the solution is expressed as $g(t)=c(t)(\varphi _{t})^{\ast
}g_{0}$, where $c(t)$ is a scaling parameter, and $\varphi _{t}$ is a
diffeomorphism.}
\end{prop}
This propositon extends directly to the pseudo-Riemannian setting. 

\begin{theorem}[\protect\cite{L01,O11}]\label{RS}
Let $\left( G,g\right) $ be a simply connected Lie group endowed
with a left-invariant pseudo-Riemannian metric $g$. If $g$ is an algebraic Ricci soliton,
then $g$ is a Ricci soliton; that is, $g$ satisfies
\eqref{Ricci soliton}  so that%
\begin{equation*}
X=\dfrac{d\varphi _{t}}{dt}\Big|_{t=0}(p)\text{ \ and }\mathrm{\exp }\left(
\frac{t}{2}\mathrm{D}\right) =d\varphi _{t}|_{e},
\end{equation*}%
where $e$ denotes the identity element of $G$ .
\end{theorem}

In the Riemannian case, Ricci solitons on homogeneous manifolds are
algebraic Ricci solitons (see \cite{J11}).
However, in the Lorentzian case there exist Ricci solitons which are not
algebraic Ricci solitons on $SL(2,\mathbb{R})$ (see \cite{BO11a}).
Recall that a Riemannian Ricci soliton on a semisimple Lie group is Einstein \cite{L01}.

It is natural to consider the opposite of Theorem \ref{RS}.
A Riemannian manifold $(M, g)$ is called a solvmanifold if there exists a
transitive solvable group of isometries.  Jablonski proved a Ricci soliton
on a solvmanifold is isometric to a solvsoliton (see \cite{J11}).  All
known examples of non-trivial Ricci solitons on homogeneous manifolds are
algebraic Ricci solitons.  In the Lorentzian case, there exists a Ricci
soliton which is not an algebraic Ricci soliton.  Indeed, Batat and Onda
proved that there exist Lorentzian Ricci solitons which are not algebraic
Ricci solitons on $SL(2,\mathbb{R})$.


It is known that all left-invariant pseudo-Riemannian metrics on $H_3$ are 
algebraic Ricci solitons.
We consider generalizations of $H$-type, namely pseudo$H$-type and 
$pH$-type.
One class of examples of pseudo$H$-type is the generalized Heisenberg 
groups $H(p,1)$ of dimension $2 p+1$ with negative definite center of
dimension $1$.
We can consider only pseudo$H$-type with nondegenerate center.  On the
other hand, $pH$-type allows both nondegenerate and degenerate
centers.  All $H (p,1)$ are also of $pH$-type.

Our main theorems are these. 
\begin{theorem} 
If $(N, \langle\cdot,\cdot\rangle )$ is of pseudo$H$-type with nondegenerate center,
then $(N, \langle\cdot,\cdot\rangle )$ has a nilsoliton.
\end{theorem}

\begin{theorem}\label{maintheorem}
Let $(N, \langle\cdot,\cdot\rangle )$ be of p$H$-type with a left-invariant Lorentzian metric. 
If $(N, \langle\cdot,\cdot\rangle )$ is a nilsoliton,  
then $(N, \langle\cdot,\cdot\rangle )$ satisfies one of following. 
\begin{itemize}
\item
If $(N, \langle\cdot,\cdot\rangle )$ has a nondegenerate center with the 
negative part, then the dimension of the center is $1$.
\item
If $(N, \langle\cdot,\cdot\rangle )$ has a nondegenerate and positive 
definite center, then $\dim N =3$.
\item
If $(N, \langle\cdot,\cdot\rangle )$ has a degenerate center, then $\dim N 
=3$.
\end{itemize} 
\end{theorem}

The paper is organized in the following way. 
In Section 2 we collect some basic facts concerning $2$-step nilpotent Lie 
groups equipped with left-invariant Lorentzian metrics. 
We introduce pseudo$H$-type and $pH$-type. 
In Section 3 we consider nilsolitons of pseudo$H$-type and $pH$-type 
and prove the main theorem. 

\section{2-step Nilpotent Lie groups}
Let $(N, g=\langle \cdot, \cdot\rangle, {\mathfrak n} )$ be a 2-step 
nilpotent Lie group with a left-invariant pseudo-Riemannian metric $g$ and
Lie algebra ${\mathfrak n}$.
We decompose ${\mathfrak n} $ as in Section 2 of \cite{CP99}.
$${\mathfrak n} = {\mathfrak z} \oplus {\mathfrak v} = {\mathfrak U}
\oplus {\mathfrak Z} \oplus {\mathfrak V} \oplus {\mathfrak E}$$
where ${\mathfrak z} = {\mathfrak U} \oplus {\mathfrak Z} $ is the center
of ${\mathfrak n} $ with degenerate part ${\mathfrak U}$.

\begin{df}
We define the involution $\iota $ by
\begin{equation*}
\iota = \left(
\begin{array}{cccc}
O & O  & I & O\\
O & E _m   & O & O\\
I & O  & O & O\\
O & O  & O & \bar{E}_n
\end{array}%
\right),
\end{equation*}
where $E_m = \mathrm{diag}(\varepsilon_1, \dots, \varepsilon_m)$, and
$\bar{E}_m = \mathrm{diag}(\bar{\varepsilon}_1, \dots, \bar{\varepsilon}_n)$.
We define the operator $J$ as
$$J_y = ad '_{\bullet} y,$$
and the operator $j$ as
$$j(y) = \iota ad '_{\bullet} \iota y, $$
for $y\in $ ${\mathfrak n}$, where $ad'$ denotes the adjoint of $ad$ with
respect to $g$.
\end{df}\noindent
The following are well known \cite{CP99}. 
\begin{prop}\label{prop_J}
We have 
$$\iota ^2 = Id$$
$$\langle x, \iota y \rangle = \langle \iota x, y\rangle $$
$$\langle J_z x, y \rangle + \langle x, J_z y\rangle = 0$$
$$\langle j (z) x, \iota y \rangle + \langle \iota x, j (z) y\rangle = 0$$
\end{prop}
The next result was proved by Cordero and Parker.
\begin{theorem}[\cite{CP99}]\label{ricci2}
Let $N$ be a $2$-step nilpotent Lie group with a left invariant metric
tensor $\langle \cdot , \cdot \rangle$ 
and Lie algebra ${\mathfrak n} = {\mathfrak U} \oplus {\mathfrak Z}
\oplus {\mathfrak V} \oplus {\mathfrak E}$ decomposed as in Section 2
of \cite{CP99}.
For all $u\in {\mathfrak U}$, $z, z' \in {\mathfrak Z}$, $ v, v' \in
{\mathfrak V}$, $e\in {\mathfrak E}$, and $x \in {\mathfrak V} \oplus
{\mathfrak E}$, we have
\begin{eqnarray*}
\Rc (u, \bullet ) &  = & 0 \\
\Rc (z, z' ) &  = & \dfrac{1}{4} \sum_a^n \bar{\varepsilon}_a \langle
   j (\iota z) e_a, j (\iota z') e_a\rangle ,  \\
\Rc (v, z ) &  = & \dfrac{1}{4} \sum_a^n \bar{\varepsilon}_a \langle j (\iota v )  e_a, j (\iota z) e_a\rangle ,  \\
\Rc (e, z) &  = & 0 \\
\Rc (v, v' ) &  = & -\dfrac{1}{2} \sum_{\alpha}^m \varepsilon _{\alpha}
   \langle j (z_{\alpha} ) v, j (z_{\alpha}) v'\rangle + \dfrac{1}{4}
   \sum_a^n \bar{\varepsilon}_a \langle j (\iota v) e_a, j (\iota v') e_a\rangle,  \\
\Rc (x, e ) &  = & -\dfrac{1}{2} \sum_{\alpha}^m \varepsilon _{\alpha}
   \langle j (z_{\alpha}) x, j (z_{\alpha}) e\rangle ,
\end{eqnarray*}
where $\varepsilon _{\alpha}  = \langle z_{\alpha}, z_{\alpha} \rangle$
and $\bar{\varepsilon}_a = \langle e_a, e_a\rangle$.
\end{theorem}


\subsection{$H$-type}

\begin{df}
Let $(N, \langle\cdot,\cdot\rangle )$ be a 2-step nilpotent Lie group with a 
left-invariant Riemannian metric.
If $({\mathfrak n}, \langle\cdot,\cdot\rangle )$ satisfies
$$J_z^2=-\langle z, z\rangle I$$
for all $z\in {\mathfrak z}$, 
then $(N, \langle\cdot,\cdot\rangle )$ is said to be of $H$-type.
\end{df}
The 3-dimensional Heisenberg group and all higher-dimensional Heisenberg
groups are of $H$-type.
The next formulas are easily checked. 





\begin{prop}\label{prop_H}
If $({\mathfrak n}, \langle \cdot,\cdot \rangle )$ is of p$H$-type, 
then the following identities hold:
\begin{eqnarray*}
\langle J_z e, J_z e\rangle & = & \langle z, z\rangle \langle e, e\rangle, \\
\langle J_z e, J_z e'\rangle & = & \langle z, z\rangle \langle e, e'\rangle, \\
\langle J_z e, J_{z'} e\rangle & = & \langle z, z'\rangle \langle e, e\rangle, \\
J_z J_{z'} + J_{z'} J_{z} & = & -2 \langle z, z'\rangle I. 
\end{eqnarray*}
\end{prop}

\subsection{pseudo$H$-type}
We consider one generalization of $H$-type.
\begin{df}
Let $(N, \langle\cdot,\cdot\rangle )$ be 2-step nilpotent Lie group with a left-invariant
pseudo-Riemannian metric and nondegenerate center.
If $({\mathfrak n}, \langle\cdot,\cdot\rangle )$ satisfies 
$$J_z^2 = \langle z,z\rangle I$$ 
for all $z\in {\mathfrak z}$,  
then $(N, \langle\cdot,\cdot\rangle )$ is said to be of pseudo$H$-type. 
\end{df}
We remark that a definition of pseudo$H$-type for a degenerate center is not
known.
The 3-dimensional Heisenberg group with non-null center, and the 
generalized Heisenberg groups $H(p,1)$ with negative-definite centers
and left-invariant Lorentzian metrics, are of pseudo$H$-type. 






The following was proved by Ciatti. 
\begin{prop}[\cite{Ciatti}]\label{prop_pseudoH}
Let $({\mathfrak n}, \langle\cdot,\cdot\rangle )$ be of pseudo$H$-type. We have
\begin{eqnarray*}
\langle J_z e, J_z e\rangle & = & -\langle z, z\rangle \langle e, e\rangle, \\
\langle J_z e, J_z e'\rangle & = & -\langle z, z\rangle \langle e, e'\rangle, \\
\langle J_z e, J_{z'} e\rangle & = & -\langle z, z'\rangle \langle e, e\rangle, \\
J_z J_{z'} + J_{z'} J_{z} & = & 2 \langle z, z'\rangle I. 
\end{eqnarray*}
for all  $z, z'\in {\mathfrak Z}$ and $e, e'\in {\mathfrak E}$.
\end{prop}

\subsection{$pH$-type}
Here is another generalization of $H$-type.
\begin{df}
Let $(N, \langle\cdot,\cdot\rangle )$ be 2-step nilpotent Lie group with a 
left-invariant pseudo-Riemannian
metric. If $({\mathfrak n}, \langle\cdot,\cdot\rangle )$ satisfies
$$j (z) ^2 = -\langle z, \iota z\rangle I, $$
then $(N, \langle\cdot,\cdot\rangle )$ is said to be of $pH$-type \cite{CP99}.
\end{df}\noindent
The generalized Heisenberg groups $H(p,1)$ are also of $pH$-type; see 
\cite{CP99}.

The following was proved by Cordero and Parker. 
\begin{prop}[\cite{CP99}]\label{prop_pH}
Let $({\mathfrak n}, \langle\cdot,\cdot\rangle )$ be of $pH$-type.
We have
\begin{eqnarray*}
\langle j (z) e, \iota j (z) e'\rangle & = & \langle z, \iota z\rangle \langle e, \iota e'\rangle, \\
\langle j (z) e, \iota j (z') e\rangle & = & \langle z, \iota z'\rangle \langle e, \iota e\rangle, \\
j (z)  j (z') + j (z') j (z) & = & -2 \langle z, \iota z'\rangle I. 
\end{eqnarray*}
for any  $z, z'\in {\mathfrak z}$ and $e, e'\in {\mathfrak v}$.
\end{prop}

\section{Nilsolitons}

Lauret \cite{L01} stated that all $H$-type have nilsolitons.
The proof is similar to that of Theorem \ref{maintheorem1}. 

\subsection{Nilsolitons of pseudo$H$-type}
We prove that most any pseudo$H$-type $N$ has a nilsoliton. 
This is similar to the Riemannian $H$-type case.  

\begin{theorem}\label{maintheorem1}
If $(N, \langle\cdot,\cdot\rangle )$ is of pseudo$H$-type with nondegenerate center,
then $(N, \langle\cdot,\cdot\rangle )$ has a nontrivial nilsoliton.
\end{theorem}

\begin{proof}
Let $(N, \langle\cdot,\cdot\rangle )$ be of pseudo$H$-type with nondegenerate center.
Then the Lie algebra is decomposed as ${\mathfrak n} = {\mathfrak Z}
\oplus {\mathfrak E}$ (${} = \mathfrak{z}\oplus\mathfrak{v}$ in Eberlein),
and the Ricci tensor is given by
\begin{eqnarray*}
\Rc (z, z' ) &  = & -\dfrac{n}{4} \langle z, z'\rangle ,  \\
\Rc (e, z) &  = & 0 \\
\Rc (e, e' ) &  = & \dfrac{m}{2} \langle e, e' \rangle,
\end{eqnarray*}
where $z, z'\in {\mathfrak Z} = \lsp z_1, \dots, z_m\rsp$ and $e,
e'\in {\mathfrak E} = \lsp e_1, \dots, e_n\rsp$.
So, the Ricci operator is given by
\begin{equation*}
\mathrm{Ric} = \left(
\begin{array}{cc}
-\dfrac{n}{4} I_m& O  \\
O & \dfrac{m}{2} I_n
\end{array}%
\right) .
\end{equation*}
\begin{lm}
If $D$ satisfies
\begin{equation}\label{derivation}
    \begin{aligned}
D z & = -\Big( m + \frac{n}{2}\Big) z, \\
D e & = -\Big(\frac{m}{2} + \frac{n}{4}\Big) e,
    \end{aligned}
\end{equation}
then $D \in \mathrm{Der}(\mathfrak{n})$.
\end{lm}
Let $D$ be the derivation (\ref{derivation}).
Put $c=m+\dfrac{n}{4}$.
It is easy to check that $(N, \langle\cdot,\cdot\rangle )$ satisfies
equation (\ref{Soliton}).  
Therefore this is a nilsoliton.
\end{proof}

\subsection{Nilsolitons of $pH$-type with nondegenerate center}
Next, we consider $pH$-type with a nondegenerate center and nilsolitons.
In this subsection we consider only Lorentzian metrics with nondegenerate center, 
and the center has the negative part $\langle z_1, z_1\rangle =-1$.
Since $j (a) x \in {\mathfrak E}$ for any $a\in {\mathfrak Z}$ and $x\in {\mathfrak E}$,
 we have
\begin{eqnarray*}
\Rc (z, z' ) &  = & \dfrac{1}{4} \sum_a^n \bar{\varepsilon}_a \langle  j (\iota z) e_a, j (\iota z') e_a\rangle ,  \\
 &  = & \dfrac{1}{4} \sum_a^n \bar{\varepsilon}_a \langle  j (\iota z) e_a, \iota j (\iota z') e_a\rangle ,  \\
 &  = & \dfrac{1}{4} \sum_a^n \bar{\varepsilon}_a \langle z, \iota z'\rangle \langle  e_a, \iota e_a\rangle ,  \\
 &  = & \dfrac{n}{4}  \langle z, \iota z'\rangle, \\
\Rc (e, e' ) &  = & -\dfrac{1}{2} \sum_{\alpha}^m \varepsilon _{\alpha}  \langle j (z_{\alpha}) e, j (z_{\alpha}) e'\rangle , \\
 &  = & -\dfrac{1}{2} \sum_{\alpha}^m \varepsilon _{\alpha} \langle e, \iota e'\rangle, \\
 &  = & -\dfrac{m-2}{2} \langle e, \iota e'\rangle.
\end{eqnarray*}
The Ricci operator is given by
\begin{itemize}
\item
if $m\ge 2$,
\begin{equation*}
\mathrm{Ric} = \left(
\begin{array}{ccc}
-\dfrac{n}{4} & O & O  \\
O & \dfrac{n}{4} I_{m-1} & O  \\
O & O & -\dfrac{m-2}{2} I_n
\end{array}%
\right) .
\end{equation*}
\item
if $m=1$,
\begin{equation*}
\mathrm{Ric} = \left(
\begin{array}{cc}
-\dfrac{n}{4} & O \\
O & \dfrac{1}{2} I_n
\end{array}%
\right) .
\end{equation*}
\end{itemize}

We obtain
\begin{theorem}\label{maintheorem2}
Let $(N\langle\cdot,\cdot\rangle )$ be of $pH$-type with nondegenerate 
center, and the center has
the negative part $\langle z_1, z_1\rangle =-1$.
\begin{itemize}
\item
If $m=1$, then $(N\langle\cdot,\cdot\rangle )$ has a nontrivial nilsoliton.
\item
If $m\ge 2$, then $(N\langle\cdot,\cdot\rangle )$ does not have a 
nontrivial nilsoliton.
\end{itemize}
\end{theorem}

\begin{proof}
Suppose that $m=1$.
\begin{lm}
If $D$ satisfies
\begin{equation}\label{derivation2}
    \begin{aligned}
D z & = \Big( -1 - \frac{n}{2}\Big) z, \\
D e & = \Big(-\frac{1}{2} - \frac{n}{4}\Big) e,
    \end{aligned}
\end{equation}
then $D \in \mathrm{Der}(\mathfrak{n})$.
\end{lm}\noindent
Let $D$ be the derivation (\ref{derivation2}). 
Put $c= 1+\dfrac{n}{4}$.
It is easy to check that $(N, \langle\cdot,\cdot\rangle )$ satisfies
equation (\ref{Soliton}).  

Next, suppose $m\ge 2$. \nopagebreak  Put
\begin{equation*}
D := \mathrm{Ric} -c \mathrm{Id} = \left(
\begin{array}{ccc}
 -\dfrac{n}{4} -c & O & O  \\
O & \Big( \dfrac{n}{4} -c\Big) I_{m-1} & O  \\
O & O & \Big(-\dfrac{m-2}{2} -c\Big) I_n
\end{array}%
\right) .
\end{equation*}
Since $(N\langle\cdot,\cdot\rangle )$ is $pH$-type, we have $[{\mathfrak
n}, {\mathfrak n}] = \lsp z_1, \dots, z_m\rsp$.
Suppose $[e, e'] = z_1$ and $[e'', e'''] = z_2$.
From $[e, e'] = z_1$, if $D \in \mathrm{Der}(\mathfrak{n})$ we obtain
$$c = -(m-2) + \dfrac{n}{4}. $$
From $[e'', e'''] = z_2$, if $D \in \mathrm{Der}(\mathfrak{n})$ we get
$$c = -(m-2)  -\dfrac{n}{4}. $$
Therefore $D \in \mathrm{Der}(\mathfrak{n})$ if and only if
$$c = -(m-2) ,  \, n = 0.$$
This happens if and only if $N$ is commutative.
\end{proof}

\begin{remark} 
If $\mathfrak{E}$ is positive-definite and the center is nondegenerate,
we calculate the Ricci operator and obtain
\begin{equation*}
\mathrm{Ric} = \left(
\begin{array}{ccc}
-\dfrac{n}{4} I_{m_q} & O & O  \\
O & \dfrac{n}{4} I_{m_p} & O  \\
O & O & -\dfrac{m-2}{2} I_n
\end{array}%
\right) ,
\end{equation*}
where $\varepsilon _1 = \cdots = \varepsilon _q = -1$ and $\varepsilon
_{q+1} = \cdots = \varepsilon _m = 1$ and $m=p+q$.
It turns out that this is not a nontrivial nilsoliton. However, we consider 
only the Lorentzian case.
\end{remark}




\subsection{Nilsolitons of $pH$-type with positive-definite center}
Let $(N, \langle\cdot,\cdot\rangle )$ be Lorentzian, of $pH$-type, and with 
positive-definite center.

\begin{prop}\label{prop_of_dim} 
$\dim {\mathfrak Z} +1 \le \dim {\mathfrak E}$. 
\end{prop}

\begin{proof}
Let $ \{ z_1, \dots, z_m, e_1, \dots, e_n\}$ be an orthonormal basis with 
$\langle e_1, e_1\rangle =-1$.
Since $(N, \langle\cdot,\cdot\rangle )$ is of $pH$-type, we may put 
$\bar{F}_{i+1}:= j(z_i) e_1 = \sum_{i=2}^n k^i e_i$ for some constants 
$k^i$. Put 
$$\bar{e}_{i} : = \Big\langle F_i, F_i \Big\rangle ^{-\frac{1}{2}} F_i.$$
Using $\bar{e}_i$ as in Gram-Schmidt orthonormalization, we can construct a 
new orthonormal basis $ \lsp z_1, \dots, z_m, e_1, \tilde{e}_2\dots,
\tilde{e}_n\rsp$. 
We now write $e_i$ for $\tilde{e}_i$. By the preceding, for any $\alpha$ 
there exists $\alpha_1$ such that 
$$j (z_\alpha) e_1 = e_{\alpha_1},$$ 
$$j (z_\alpha) e_{\alpha_1} = -e_1,$$ 
and $j (z_\alpha) e_{\beta}$ does not have an $e_1$ component for $\beta \ne 
\alpha$. Let $m (i)$ be the number of $\bar{\imath}$ such that  
$$j (z_{\bar{\imath}}) e_1 = e_{i},$$ 
$$j (z_{\bar{\imath}}) e_{i} = -e_1,$$ 
for $i = 1, \dots, n$. 
\begin{lm}
$m(1)=0$ and $m(i)=0$ or $1$. 
\end{lm}
\begin{proof}
From the definition, it follows that $m(1)=0$. 
Let $m(2)=2$. 
Then we have $\alpha_2, \bar{\alpha}_2$ such that 
$$j (z_{\alpha_2}) e_1 = e_{2},$$ 
$$j (z_{\alpha_2}) e_{2} = -e_1, $$ 
and 
$$j (z_{\bar{\alpha}_2}) e_1 = e_{2},$$ 
$$j (z_{\bar{\alpha}_2}) e_{2} = -e_1.$$ 
Then we obtain 
$$\langle j (z_{\alpha_2}) e_{1} , \iota j (z_{\bar{\alpha}_2}) e_{1} \rangle = \langle e_2,  \iota e_2\rangle=1. $$
On the other hand, from Proposition \ref{prop_pH} we find 
$$\langle j (z_{\alpha_2}) e_{1} , \iota j (z_{\bar{\alpha}_2}) e_{1} \rangle = \langle z_{\alpha_2}, \iota z_{\bar{\alpha}_2}\rangle \langle e_1, \iota e_1\rangle = \langle z_{\alpha_2}, \iota z_{\bar{\alpha}_2}\rangle .$$
Therefore $z_{\bar{\alpha}_2} = z_{\alpha_2}$. 
\end{proof}
So, we change the order of our new orthonormal basis $ \{ z_1, \dots, 
z_m, e_1, \dots, e_n\}$ so that $\langle e_1, e_1\rangle =-1$ and 
$$j (z_\alpha) e_1 = e_{\alpha +1} ,$$ 
$$j (z_\alpha) e_{\alpha +1} = -e_1 .$$ 
If $m+1 = \dim {\mathfrak Z} +1> \dim {\mathfrak E} =n$, 
we have $j(z_n) e_1 = \sum_{a=2}^n h^a e_a$ for some constants $h^a$. 
Then for any $i=1, \dots, n-1$ we obtain 
$$\langle j (z_n) e_{1} , \iota j (z_i) e_{1} \rangle = \langle \sum_{a=2}^n
h^a e_a,  \iota e_{a+1}\rangle=h^{i+1}.$$
On the other hand, from Proposition \ref{prop_pH} we get 
$$\langle j (z_n) e_{1} , \iota j (z_i) e_{1} \rangle = \langle z_n, \iota z_i \rangle \langle e_1, \iota e_1\rangle = 0.$$
Thus $j(z_n) e_1=0$. This is a contradiction. 
\end{proof}

\begin{theorem}\label{maintheorem3} 
$(N, \langle\cdot,\cdot\rangle )$ is a nilsoliton if and only if $m=1$ and $n=2$. 
\end{theorem}

\begin{proof} 
We use the orthonormal basis in Proposition \ref{prop_of_dim}. 
From Theorem \ref{ricci2}, we obtain the nonzero components of the Ricci 
tensor as
\begin{eqnarray*}
\Rc (z_i, z_i) &  = & \dfrac{1}{4} \sum_a^n \bar{\varepsilon}_a \langle j (\iota z_i) e_a, j (\iota z_i) e_a\rangle ,  \\
   &  = & \dfrac{1}{4} \Big\{ \bar{\varepsilon}_1 \langle j (z_i) e_1, j (z_i) e_1\rangle  
+ \bar{\varepsilon}_{i_1} \langle j (z_i) e_{i_1}, j (z_i) e_{i_1}\rangle \\  
& & + \sum_{a\ne 1, i_1} \bar{\varepsilon}_a \langle j (z_i) e_a, j (z_i) e_a\rangle \Big\} \\
   &  = & \dfrac{1}{4} (n-4), \\
\Rc (e_i, e_i) &  = & -\dfrac{1}{2} \sum_{\alpha \ne 1, i_1} \langle j (z_{\alpha }) e_1, j (z_{\alpha }) e_1\rangle = -\dfrac{m-2 m (i)}{2}, \\
\end{eqnarray*}
for $i= 1, \dots, n$. 
The Ricci operator is then given by
$$
\mathrm{Ric} = \mathrm{diag} \left( \dfrac{1}{4} (n-4), -\dfrac{m -2 m (1)}{2}, \dots, -\dfrac{m -2 m (n)}{2}\right) .
$$
Put
\begin{equation*}
D := \mathrm{Ric} -c \cdot \mathrm{Id}. 
\end{equation*} 
For any $i,j$, 
we can put $[e_i, e_j] = z_{i j}$. 
If $D \in \mathrm{Der}(\mathfrak{n})$ we find 
$$\dfrac{1}{4} (n-4)-c = 2 m - 2 m(i) -2 m(j) -2 c. $$
If $n\ne 2$, this is incompatible. 
If $n=2$, then $m=1$ and $(N, \langle\cdot,\cdot\rangle )$ is the three-dimensional Heisenberg group. 
We obtain $c = \dfrac{1}{2}$. 
This is a nilsoliton (see also \cite{O11}). 
\end{proof} 

\subsection{Nilsolitons of $pH$-type with degenerate center}
For Lorentzian nilsolitons of $pH$-type with degenerate center, we have a
basis $\{ u, z_1, \dots, z_m, v, e_1, \dots, e_n\}$ with $\varepsilon_i=1$
for any $i=1, \dots, n$.
Since this is of $pH$-type, it is easy to check that $n$ is odd. 
Put $n=2 k+1$. 
Now we prove
\begin{theorem}
This is a nilsoliton if and only if $n=1$. 
Moreover, we obtain the flat metric on the $3$-dimensional Heisenberg group. 
\end{theorem} 

\begin{proof}
As in the proof of Proposition \ref{prop_of_dim}, 
we construct a new basis 
$$\{ u, z_1, \dots, z_m, v, e_1, \dots, e_n\}$$ 
with
$$j (u) v = e_1,$$ 
$$j (u) e_1 = -v,$$ 
$$j (z_i) v = e_{i+1} ,$$ 
$$j (z_i) e_{i+1} = -v.$$ 
From Theorem \ref{ricci2}, the Ricci tensor is given by
\begin{eqnarray*} 
\Rc (u, \bullet ) &  = & 0 \\
\Rc (z_{\alpha}, z_{\beta} ) &  = & \dfrac{1}{4} \sum_a^n \bar{\varepsilon}_a \langle j(\iota z_{\alpha}) e_a, j(\iota z_{\beta}) e_a\rangle ,  \\
&  = & \dfrac{1}{4} \sum_a^n \langle j(\iota z_{\alpha}) e_a, \iota j(\iota z_{\beta}) e_a\rangle ,  \\
&  = & \dfrac{1}{4} \sum_a^n \langle z_{\alpha}, \iota z_{\beta} \rangle \langle e_a, \iota e_a\rangle,  \\
&  = & \dfrac{n}{4} \langle z_{\alpha}, \iota z_{\beta} \rangle \\
\Rc (v, z_{\alpha}) & = & \dfrac{1}{4} \sum_a^n \bar{\varepsilon}_a \langle j(\iota v) e_a, j(\iota z_{\alpha}) e_a\rangle ,  \\
&  = & \dfrac{1}{4} \sum_a^n \langle j(u) e_a, j(\iota z_{\alpha}) e_a\rangle ,  \\
&  = & \dfrac{1}{4} \Big\{ \langle j(u) e_1, j(\iota z_{\alpha}) e_1\rangle + \sum_{a=2}^n \langle j(u) e_a, j(\iota z_{\alpha}) e_a\rangle \Big\},  \\
&  = & 0 \\ 
\Rc (v, v) &  = & -\dfrac{1}{2} \sum_{\alpha}^m \varepsilon _{\alpha}
   \langle j (z_{\alpha} ) v, j (z_{\alpha}) v\rangle + \dfrac{1}{4}
   \sum_a^n \bar{\varepsilon}_a \langle j (\iota v) e_a, j (\iota v) e_a\rangle,  \\
 &  = & -\dfrac{1}{2} \sum_{\alpha}^m \langle e_{\alpha}, e_{\alpha} \rangle 
+ \dfrac{1}{4} \sum_{a=2}^n \langle v, \iota v \rangle \langle e_a, \iota e_a\rangle,  \\   
&  = & \dfrac{n-2 m-1}{4}  \, \, (\mathrm{if}\ n\ge 2)\\
\Rc (v, v) & = & 0 \, \, (k=1), \\
\Rc (v, e_i) & = & 0, \\
\Rc (e_1, e_b) &  = & -\dfrac{1}{2} \sum_{\alpha}^m \varepsilon _{\alpha} \langle j (z_{\alpha}) e_1, j (z_{\alpha}) e_b\rangle , \\
&  = & -\dfrac{m}{2} \delta _{1 b} \\
\Rc (e_a, e_b) &  = & -\dfrac{m-1}{2} \delta _{a b} \, \, (a, b\ge 2)\\
\end{eqnarray*}
If $n=1$, then we obtain the flat metric. 
Suppose $n\ge 2$. 
Then we get the Ricci operator as 
\begin{equation*}
\mathrm{Ric} = \left(
\begin{array}{ccccc}
0 & 0 & -\dfrac{n-2 m-1}{4} & 0 & O  \\
0 & \dfrac{n}{4} I_m & O & O & O  \\
O & O & O & O & O  \\
O & O & O & -\dfrac{m}{2} & O  \\
O & O & O & O & -\dfrac{m-1}{2} I_n  \\
\end{array}%
\right) . 
\end{equation*}
Put
\begin{equation*}
D := \mathrm{Ric} -c \cdot \mathrm{Id}. 
\end{equation*} 
From $[v, e_1] = u$, 
if $D \in \mathrm{Der}(\mathfrak{n})$ we find
 $$c = -\dfrac{m}{2}.$$
From $[v, e_2,] = z_1$, 
if $D \in \mathrm{Der}(\mathfrak{n})$ we find
 $$c = -\dfrac{m-1}{2}.$$
This is incompatible. 
\end{proof}

\subsection{Nilsolitons of $pH$-type with degenerate center in the 
pseudo-Riemannian case}

Next, we consider a degenerate center with a pseudo-Riemannian metric.
It seems we cannot calculate the Ricci tensor of $pH$-type with degenerate
center in any general or systematic way suitable for the present use.  We
can see examples that are nilsolitons and others that are not.

\begin{ex}
Let $(H_3, g_3)$ be the 3-dimensional Heisenberg group with degenerate
center. This is of $pH$-type (see \cite{CP99}) and flat (see
\cite{N79} and \cite{CP99}).
\end{ex}

\begin{ex}
We consider Example 2.6 (and 3.30) in \cite{CP99}, a $7$-dimensional
Heisenberg group with the simplest quaternionic Heisenberg algebra.  We
take a basis $\{ u_1,u_2, z, v_1, v_2, e_1, e_2\}$ with structure
equations
\begin{eqnarray*}
&&%
\begin{array}{c}
 [e_1, e_2] = z \hspace{2cm}  [v_1, v_2] = z
\end{array}
\notag \\
&&%
\begin{array}{c}
 [e_1, v_1] = u_1 \hspace{1.82cm}  [e_2, v_1] = u_2
\end{array}
\notag \\
&&%
\begin{array}{c}
 [e_1, v_2] = u_2 \hspace{1.82cm}  [e_2, v_2] = -u_1
\end{array}
\notag \\
\end{eqnarray*}
and nontrivial inner products
$$\langle u_i, v_j\rangle = \delta _{i j}, \, \langle z, z\rangle = \varepsilon, \, \langle e_a, e_a\rangle = \bar{\varepsilon} _a.$$

\begin{theorem}
This is not a nilsoliton.
\end{theorem}

\begin{proof}
The Ricci operator is given by
\begin{equation*}
\mathrm{Ric} = \dfrac{1}{2} \varepsilon \bar{\varepsilon}_1 \bar{\varepsilon}_2 \left(
\begin{array}{cccc}
O & & & \\
 & 1  & & \\
 & & O & \\
& & & -I_2
\end{array}%
\right) .
\end{equation*}
Put
\begin{equation*}
D := \mathrm{Ric} -c \cdot \mathrm{Id} = \left(
\begin{array}{cccc}
-c I_2 & & & \\
 & \dfrac{1}{2} \varepsilon \bar{\varepsilon}_1 \bar{\varepsilon}_2  -c  & & \\
 & & -c I_2 & \\
& & & -\Big( \dfrac{1}{2} \varepsilon \bar{\varepsilon}_1 \bar{\varepsilon}_2  + c\Big) I_2
\end{array}%
\right) .
\end{equation*}
From $[e_1, e_2] = z$, if $D \in \mathrm{Der}(\mathfrak{n})$ we obtain
$$c = -\dfrac{3}{2} \varepsilon \bar{\varepsilon}_1 \bar{\varepsilon}_2. $$
From $[e_1, v_1] = u_1$, if $D \in \mathrm{Der}(\mathfrak{n})$ we obtain 
$$c = -\dfrac{1}{2} \varepsilon \bar{\varepsilon}_1 \bar{\varepsilon}_2. $$

Therefore $D \in \mathrm{Der}(\mathfrak{n})$ if and only if
$$c = \varepsilon \bar{\varepsilon}_1 \bar{\varepsilon}_2 = 0.$$
But this is impossible.
\end{proof}

\end{ex}

Kensuke Onda

Graduate School of Mathematics, Nagoya University, Furocho, Chikusaku, 
\newline
\hspace{1.2em} Nagoya 464-8602 JAPAN 

Tel: +81-052-789-2429 / FAX: +81-052-789-2829

\textit{E-mail address}: kensuke.onda@math.nagoya-u.ac.jp

\bigskip

Phillip E. Parker

Mathematics Department, Wichita State University, 
\newline 
\hspace{1.2em} Wichita KS 67260-0033 USA

\textit{E-mail address}: phil@math.wichita.edu

\end{document}